\documentclass[11pt]{article}

\usepackage[T1]{fontenc}
\usepackage{newtxtext,newtxmath}
\usepackage[final]{microtype}
\usepackage[margin=1in]{geometry}
\usepackage{authblk}

\usepackage[authoryear,round]{natbib}
\usepackage{amsmath,amsthm,mathtools}
\usepackage{booktabs,array,graphicx,siunitx}
\usepackage[font=small,labelfont=bf,labelsep=period]{caption}
\usepackage{algorithm,algpseudocode}
\usepackage{xurl,placeins}
\usepackage{xr}
\usepackage[colorlinks=true,allcolors=blue]{hyperref}
\usepackage[capitalize,nameinlink]{cleveref}
\theoremstyle{plain}
\newtheorem{theorem}{Theorem}
\newtheorem{proposition}[theorem]{Proposition}
\newtheorem{lemma}[theorem]{Lemma}
\newtheorem{corollary}[theorem]{Corollary}
\theoremstyle{definition}

\theoremstyle{remark}

\newcommand{\R}{\mathbb R}
\newcommand{\B}{\mathcal B}

\newcommand{\pos}[1]{[#1]_+}
\newcommand{\keywords}[1]{\par\medskip\noindent\textbf{Keywords:} #1}

\makeatletter
\@ifundefined{ifblind}{\newif\ifblind\blindfalse}{}
\@ifundefined{blindsubmission}{}{\blindtrue}
\@ifundefined{ifcompanion}{\newif\ifcompanion\companionfalse}{}
\@ifundefined{ifincludeappendix}{\newif\ifincludeappendix\includeappendixtrue}{}
\makeatother

\ifdefined\opresubmission
  \usepackage{setspace}
  \ifcompanion\else
    \usepackage[tablesonly,nomarkers,nolists]{endfloat}
    
  \fi
\fi

\ifcompanion
  \title{Electronic Companion to ``Simultaneous Group-Envelope Bounds for $\Gamma$-Robust Multiple-Choice Knapsack Problems''}
\else
  \title{Simultaneous Group-Envelope Bounds for $\Gamma$-Robust Multiple-Choice Knapsack Problems}
\fi
\ifblind
  \author{Anonymous submission}
  \affil{}
  \hypersetup{pdfauthor={Anonymous}}
\else
  \author{Zi Yuan Eric Shao}
  \affil{Department of Mathematics, ETH Z\"urich, \texttt{ershao@student.ethz.ch}}
  \hypersetup{pdfauthor={Zi Yuan Eric Shao}}
\fi
\ifblind
  \date{}
\else
  \date{August 2026}
\fi
\hypersetup{pdftitle={Simultaneous Group-Envelope Bounds for Gamma-Robust Multiple-Choice Knapsack Problems}}

\newcommand{\PubValidationMaxError}{\num{6.58e-08}}
\newcommand{\PubValidationMinSlack}{\num{4.26e-14}}
\newcommand{\PubValidationCertViolation}{\num{3.98e-13}}
\newcommand{\PubValidationCertGap}{\num{9.92e-09}}
\newcommand{\PubKernelCases}{48}
\newcommand{\PubKernelLargeSpeedup}{60.14}
\newcommand{\PubKernelMaxError}{\num{4.27e-07}}
\newcommand{\PubKernelCompressedSlope}{0.96}
\newcommand{\PubKernelDenseSlope}{2.03}
\newcommand{\PubTraceCases}{24}
\newcommand{\PubTraceIntervals}{168}
\newcommand{\PubTraceGeoSpeedup}{2.21}
\newcommand{\PubTraceWins}{24}
\newcommand{\PubTraceDominancePct}{100.0\%}
\newcommand{\PubTraceCertGap}{\num{9.67e-09}}
\newcommand{\PubPrimaryCases}{60}
\newcommand{\PubPrimaryGeoSpeedup}{2.37}
\newcommand{\PubPrimaryCILow}{2.27}
\newcommand{\PubPrimaryCIHigh}{2.47}
\newcommand{\PubPrimaryWins}{60}

\newcommand{\PubPrimaryCertGap}{\num{9.53e-09}}
\newcommand{\PubPrimaryMedianThetaPct}{28.8\%}

\newcommand{\PubRepeatBlocks}{300}
\newcommand{\PubRepeatWins}{300}
\newcommand{\PubRepeatMinSpeedup}{1.44}
\newcommand{\PubCompressedMedianCV}{0.58\%}
\newcommand{\PubCliqueMedianCV}{0.44\%}
\newcommand{\PubCliqueMedianNnz}{\num{100394}}
\newcommand{\PubCliqueMedianCSRMiB}{1.19}

\newcommand{\PubFamilyDense}{4.26}
\newcommand{\PubFamilyBreakpoints}{1.77}

\newcommand{\PubStressGeoSpeedup}{5.36}
\newcommand{\PubStressNHighSpeedup}{6.28}
\newcommand{\PubApplicationCases}{9}
\newcommand{\PubApplicationWins}{0}
\newcommand{\PubApplicationGeoSpeedup}{0.27}
\newcommand{\PubApplicationCILow}{0.24}
\newcommand{\PubApplicationCIHigh}{0.32}
\newcommand{\PubApplicationNHighSpeedup}{0.33}
\newcommand{\PubExternalCases}{9}
\newcommand{\PubExternalWins}{9}
\newcommand{\PubExternalGeoSpeedup}{2.75}

\newcommand{\PubExternalNHighSpeedup}{4.00}
\newcommand{\PubExternalCertGap}{\num{9.85e-09}}

\begin{document}
\maketitle

\ifcompanion
\begin{center}
\emph{Comparator formulation, exact-integration audit, statistical protocol, generators, and reproducibility record.}
\end{center}
\else

\begin{abstract}
Many robust planning problems are solved by checking a family of ordinary optimization problems, one for each uncertainty threshold.  Repeatedly building and solving those relaxations can dominate runtime.  We show that, when a decision chooses exactly one option from each group, the entire threshold family of multiple-choice knapsack relaxations can instead be bounded together.  A cancellation removes threshold-specific baselines and reduces each group's contribution to two simple envelopes.  After sorting, one multiplier is evaluated across all thresholds in time linear in their number and nearly linear in the number of options.  The resulting interval bound is valid, matches the ordinary relaxation on a single feasible threshold, and is at least as strong as the matched group-clique comparator.  A certified one-dimensional search controls multiplier error, while exact checks resolve numerically ambiguous feasibility and comparison cases.  In an otherwise identical adaptive search, the method wins all 60 instance-level median comparisons and achieves a \PubPrimaryGeoSpeedup-fold geometric-mean speedup.  Additional ablations, independent linear-program checks, and nine instances transformed from a published archive test the mechanism beyond the internal generator.  A separate integer audit shows the boundary: faster relaxation bounds need not help when integer search dominates.
\end{abstract}


\noindent\textbf{Area of review:} Optimization.\\
\textbf{Subject classifications:} Programming, integer: robust multiple-choice knapsack; Programming, linear: Lagrangian upper bounds; Analysis of algorithms: parametric evaluation.

\section{Introduction}\label{sec:introduction}

Cardinality-budget uncertainty reduces a robust binary problem to finitely many nominal problems indexed by a scalar protection threshold \citep{BertsimasSim2003,BertsimasSim2004}.  The reduction is polynomial, but a conventional method may still solve one nominal relaxation for every distinct deviation.  Strong bounded-threshold formulations and divide-and-conquer reduce this burden \citep{AlvarezMiranda2013,Hansknecht2018,BusingGersingKoster2023}; the remaining task is to obtain a strong bound over an interval of thresholds without repeatedly materializing a large relaxation.

That task is practically consequential when LP relaxations, rather than the nominal integer subproblems, control runtime.  In the closest large-scale study identified in our review of bounded-threshold robust relaxations, \citet{BusingGersingKoster2023} report that LP solves consume 33.1\% of time on average among solved instances.  In their million-item robust-knapsack experiment, the share reaches 93.17\%, and the root relaxation alone averages 1,062 seconds.  Their model places uncertainty in the objective and therefore differs from the resource-uncertainty model below; the figures are motivation rather than transferable performance estimates.  They nevertheless identify the relevant computational bottleneck directly: solving the bounded-threshold LP family can dominate a robust search.

Many operational decisions have an exactly-one menu structure: one price per product, one configuration per component, or one service level per customer segment.  Every fixed-threshold relaxation of the resulting robust model is a multiple-choice knapsack linear program.  Binary and at-most-one knapsacks are included by adding a dummy option.  Although one fixed relaxation is classical and inexpensive, jointly bounding the full parameterized family is a different computational task.

The primary deliverable here is a faster complete threshold-disjunctive LP certificate.  It can be used as a root certificate or as an interval bound in an outer threshold search for an exactly-one robust model.  This is deliberately narrower than a new end-to-end integer solver: the value arises when interval LP work is material, whereas faster interval bounds cannot remove time spent inside fixed-threshold integer subproblems.

This paper identifies a cancellation created by the exactly-one equations.  After dualizing a fixed-threshold capacity constraint, the threshold-dependent group baselines disappear.  Within each local deviation interval, a group contributes the larger of a constant and a line whose slope is common to all active options.  Prefix and suffix maxima together with range accumulation therefore evaluate one multiplier simultaneously over the full global threshold grid.  The resulting algorithm is linear in the threshold count and nearly linear in the option count after sorting.

The paper makes four contributions.  First, it derives the cancellation and a two-envelope algorithm that evaluates a multiplier over all thresholds at once.  Second, it proves exact-minimax dominance over the closest bounded-threshold group-clique relaxation and supplies a convex bracketing procedure with an explicit implemented optimality gap.  Third, it embeds the certified bound in an adaptive decomposition with global-bound invariants and a separately scoped exact-integer extension.  Fourth, a fixed-design study separates algebra, multiplier certification, kernel scaling, matched-interval behavior, complete-certificate time, robustness, and external-coefficient transfer.  The comparator is deliberately component-matched: the clique relaxation and the envelope bound are placed under the same search, fixed-threshold solver, stopping rule, and timing policy.  The contribution is the simultaneous parametric bound, not a new uncertainty set, a faster algorithm for one knapsack relaxation, a reproduction of a complete DnC+ implementation, or universal superiority of an integer solver.

\ifblind
The relationship to a companion pricing paper is deliberately one-directional.  The anonymized companion reference derives the finite-menu robust-pricing model, its exact full-breakpoint MCKP decomposition, the fixed-threshold hull relaxation and one-item rounding certificate, and exact full-family search.
\else
The relationship to the author's companion pricing paper is deliberately one-directional.  \citet{Shao2026PaperA} derives the finite-menu robust-pricing model, its exact full-breakpoint MCKP decomposition, the fixed-threshold hull relaxation and one-item rounding certificate, and exact full-family search.
\fi
The present paper accepts that scalar threshold family as input and changes the computational objective: it certifies the maximum LP value over many thresholds through one-multiplier simultaneous evaluation and interval minimax bounds.  The baseline cancellation, two-envelope data structure, minimax-dominance theorem, certified multiplier search, and matched-interval experiments appear only here.  Thus Paper~A supplies the model and certifying fixed-threshold foundation; Paper~B supplies an all-threshold LP-family accelerator that can be placed on top of it.

\section{Robust MCKP and prior-art boundary}\label{sec:problem}

\subsection{The \texorpdfstring{$\Gamma$}{Gamma}-robust MCKP}

Let $i\in\{1,\ldots,n\}$ index groups and let $J_i$ be the finite, nonempty option set of group $i$.  Option $(i,j)$ has value $v_{ij}\in\R$, nominal resource contribution $a_{ij}\in\R$, and nonnegative deviation $d_{ij}\ge0$.  Exactly one option is selected from every group.  With an integer uncertainty budget $\Gamma\in\{0,\ldots,n\}$, the problem is
\begin{align}
\max_x\quad & \sum_i\sum_{j\in J_i}v_{ij}x_{ij} \label{eq:robust-mckp}\tag{R-MCKP}\\
\text{s.t.}\quad &
\sum_i\sum_{j\in J_i}a_{ij}x_{ij}
-\max_{\substack{0\le u_i\le1\\\sum_i u_i\le\Gamma}}
\sum_i u_i\sum_{j\in J_i}d_{ij}x_{ij}\ge0,\nonumber\\
&\sum_{j\in J_i}x_{ij}=1 &&(i=1,\ldots,n),\nonumber\\
&x_{ij}\in\{0,1\}.\nonumber
\end{align}
For integer $\Gamma$, the inner maximum is the sum of the $\Gamma$ largest selected deviations.  The model covers a robust resource or residual constraint; the sign convention lets larger $a_{ij}$ represent more nominal slack.

The mathematical results use finite real coefficients.  The released binary64 implementation additionally requires the aggregate objective range $[\sum_i\min_j v_{ij},\sum_i\max_j v_{ij}]$ to lie within the finite binary64 range; it checks this condition and requests objective rescaling otherwise.  This scale precondition prevents an unrepresentable LP value from being mistaken for an exact finite certificate.

For a fixed binary selection, write $d_i:=\sum_jd_{ij}x_{ij}$ for its selected deviation in group $i$.  LP duality for the inner budget problem gives
\begin{equation}\label{eq:top-gamma}
\max_{0\le u\le1,\ \mathbf 1^\top u\le\Gamma}\sum_i d_i u_i
=\min_{\theta\ge0}\left\{\Gamma\theta+\sum_i\pos{d_i-\theta}\right\}.
\end{equation}
Define the global candidate set
\[
\B:=\{0\}\cup\{d_{ij}:i=1,\ldots,n,\ j\in J_i\}
=\{b_0<\cdots<b_{B-1}\},\qquad B:=|\B|.
\]
Because the right-hand side of \eqref{eq:top-gamma} is convex piecewise linear with breakpoints at the selected deviations, a binary selection is robust feasible if and only if there is a $\theta\in\B$ such that
\begin{equation}\label{eq:fixed-threshold-feasibility}
\sum_i\sum_{j\in J_i}r_{ij}(\theta)x_{ij}\ge\Gamma\theta,
\qquad
r_{ij}(\theta):=a_{ij}-\pos{d_{ij}-\theta}.
\end{equation}
Thus \eqref{eq:robust-mckp} is a finite union of fixed-threshold MCKPs.  This classical representation is the starting point.

The exactly-one form also covers binary and at-most-one selection.  Append a dummy option with $(v,a,d)=(0,0,0)$ to an at-most-one group.  More specifically, the robust binary knapsack
\[
\max\left\{\sum_i p_i y_i:\ \sum_i w_i y_i+
\max_{0\le u\le1,\ \mathbf 1^\top u\le\Gamma}\sum_i\delta_i y_i u_i\le W,
\ y\in\{0,1\}^n\right\}
\]
is \eqref{eq:robust-mckp} with two options per item: the unselected option $(0,W/n,0)$ and the selected option $(p_i,W/n-w_i,\delta_i)$.  Hence the results below include robust binary knapsack without changing asymptotic complexity.

For completeness, threshold sufficiency follows directly from convexity.  For a fixed selected deviation vector $q$, the protection function $h(\theta)=\Gamma\theta+\sum_i\pos{q_i-\theta}$ is convex and piecewise linear.  Away from a selected deviation its slope is $\Gamma-|\{i:q_i>\theta\}|$, so a minimum is attained at zero or at a selected deviation, with a flat interval allowed.  The global set $\B$ contains all such candidates.

\subsection{Finite-menu pricing: MCKP derivation and use}\label{sec:pricing-specialization}

This specialization follows the certifying finite-menu pricing framework of \citet{Shao2026PaperA}; the derivation is repeated here only to make the input to the all-threshold oracle explicit.

Finite-menu pricing supplies one operational specialization.  Product $i$ has reference price $\bar p_i>0$, exposure $w_i\ge0$, and admissible prices $p_{ij}>0$, $j\in J_i$; admissibility can encode the reference band $p_{ij}\in[(1-\sigma_i)\bar p_i,(1+\sigma_i)\bar p_i]$ for $\sigma_i\in[0,1)$.  Binary $x_{ij}$ selects exactly one price, so products are the MCKP groups and candidate prices are their options.  At option $j$, demand is $\widetilde g_{ij}=\hat g_{ij}+\xi_i\delta_{ij}$, where $0\le\delta_{ij}\le\hat g_{ij}$, $|\xi_i|\le1$, and $\sum_i|\xi_i|\le\Gamma$.  Let $\widetilde N(x,\xi):=\sum_{i,j}w_ip_{ij}\widetilde g_{ij}x_{ij}$ and $\widetilde D(x,\xi):=\sum_{i,j}w_i\bar p_i\widetilde g_{ij}x_{ij}$.  Nominal revenue $\widetilde N(x,0)$ is maximized subject to $\widetilde N(x,\xi)/\widetilde D(x,\xi)\ge\Delta$ for every admissible $\xi$.

If $\widetilde D>0$ throughout, the robust ratio constraint is equivalent to requiring the worst-case residual $\widetilde N-\Delta\widetilde D$ to be nonnegative.  Because exactly one option is selected per product, adverse signs subtract the $\Gamma$ largest absolute selected residual deviations:
\begin{align}
\min_{|\xi_i|\le1,\ \sum_i|\xi_i|\le\Gamma}
\{\widetilde N(x,\xi)-\Delta\widetilde D(x,\xi)\}
={}&\sum_{i,j}a_{ij}x_{ij}
-\max_{\substack{0\le u_i\le1\\\sum_i u_i\le\Gamma}}
\sum_i u_i\sum_jd_{ij}x_{ij}. \label{eq:pricing-residual}
\end{align}
Together with nominal revenue as the objective, this is exactly \eqref{eq:robust-mckp}, with coefficients
\begin{equation}\label{eq:pricing-map}
v_{ij}=w_ip_{ij}\hat g_{ij},\qquad
a_{ij}=w_i(p_{ij}-\Delta\bar p_i)\hat g_{ij},\qquad
d_{ij}=\left|w_i(p_{ij}-\Delta\bar p_i)\delta_{ij}\right|.
\end{equation}
Without positivity, only the additive residual interpretation remains.

To apply the paper's results, a decision maker specifies the price menus, nominal demands, deviation scales, robustness budget $\Gamma$, and target ratio $\Delta$, then constructs \eqref{eq:pricing-map} once.  Substitution into \eqref{eq:fixed-threshold-feasibility} produces one nominal MCKP relaxation per deviation threshold.  The group-envelope oracle bounds these relaxations simultaneously, and the adaptive algorithm certifies the best LP value over the complete threshold family.  Thus the certificate supplies a root or outer-search upper bound on robust nominal revenue; it accelerates robust price-menu optimization when interval LP work is material, but it does not estimate demand or select $\Gamma$ or $\Delta$.

\FloatBarrier
\subsection{Prior-art boundary}

\Cref{tab:prior-art} separates the task studied here from its ingredients: threshold enumeration and filtering \citep{BertsimasSim2003,AlvarezMiranda2013,LeeKwon2014}, divide-and-conquer \citep{Hansknecht2018}, robust formulations and inequalities \citep{Atamturk2006,FischettiMonaci2012,Monaci2013,JoungPark2021,JoungOhLee2023}, and fixed-MCKP algorithms \citep{SinhaZoltners1979,Zemel1980,Dyer1984,Szkaliczki2025}.  The closest formulation component identified in our review is the bounded-threshold framework with clique aggregation of \citet{BusingGersingKoster2023}.  None of these sources, to our knowledge, derives the baseline cancellation, two-envelope representation, certified simultaneous multiplier evaluation, or its exact-minimax dominance relation for this threshold family.

The formulation component must be distinguished from the complete algorithms in that source.  DnC+ combines filtered threshold sets, strengthened estimators, optimality cuts, incumbent improvement, and early termination of nominal subproblems; its official Java implementation uses Gurobi and solves robust binary models with objective uncertainty.  A direct runtime comparison would therefore change the uncertainty placement, final deliverable, solver stack, and several algorithmic components simultaneously.  Instead, the theoretical and computational comparisons here isolate the bounded-threshold clique relaxation: \Cref{sec:computations} specializes it exactly to the group structure and evaluates it under the same adaptive search as the proposed oracle.  This controlled comparison answers whether the new interval-bound component is stronger and faster; it does not reinterpret that result as a whole-solver comparison with DnC+.

\begin{table}[t]
\centering
\caption{Boundary between the present task and related algorithmic streams.}
\label{tab:prior-art}
\small
\renewcommand{\arraystretch}{1.18}
\begin{tabular}{>{\raggedright\arraybackslash}p{0.24\textwidth}>{\raggedright\arraybackslash}p{0.29\textwidth}>{\raggedright\arraybackslash}p{0.37\textwidth}}
\toprule
Stream & Established operation & Treatment here \\
\midrule
Bertsimas--Sim and filtered enumeration & Solve a finite, possibly reduced set of nominal threshold problems & Simultaneously evaluates a Lagrangian multiplier over all thresholds \\\addlinespace[4pt]
Companion Paper~A \citep{Shao2026PaperA} & Derives robust pricing, the exact full-breakpoint MCKP family, fixed-threshold hull/rounding certificates, and exact full-family search & Takes that family as input and accelerates certification of its maximum LP value; does not repeat Paper~A's pricing or rounding contribution \\\addlinespace[4pt]
Fast fixed MCKP LP & Solves one MCKP relaxation using hull or partition structure & Amortizes evaluation across a parameterized family; does not improve the one-LP complexity \\\addlinespace[4pt]
Divide-and-conquer & Avoids threshold solves using monotonicity & Replaces an interval LP by a direct exactly-one group-envelope bound \\\addlinespace[4pt]
Strong and bounded-threshold formulations & Strengthen or linearize the robust model over a threshold interval & Proves objective-bound dominance over the group-clique interval LP without introducing a new robust formulation \\\addlinespace[4pt]
Ellipsoidal robust MMKP \citep{Caserta2019} & Treats multiple resources and covariance-based uncertainty & Different uncertainty set, geometry, and computational target \\
\bottomrule
\end{tabular}
\end{table}

The argument proceeds in three logical steps.  The threshold representation first converts the robust model into a finite family of fixed-threshold linear relaxations.  Lagrangian duality then replaces that family by a minimax interval certificate and permits comparison with the group-clique relaxation.  Finally, the exactly-one equations expose the cancellation that makes simultaneous evaluation computationally useful.  Keeping these steps separate is important: certificate validity follows from duality, whereas speed follows from the group-envelope data structure.

\section{A Lagrangian bound over threshold intervals}\label{sec:interval-bound}

Fix $\theta\in\B$ and define the group baseline
\[
r_i^*(\theta):=\max_{j\in J_i}r_{ij}(\theta),
\quad c_{ij}(\theta):=r_i^*(\theta)-r_{ij}(\theta)\ge0,
\quad C(\theta):=\sum_i r_i^*(\theta)-\Gamma\theta.
\]
Using the exactly-one equations, \eqref{eq:fixed-threshold-feasibility} is equivalent to
$\sum_{i,j}c_{ij}(\theta)x_{ij}\le C(\theta)$.  If $C(\theta)<0$, the fixed-threshold problem is infeasible.  If $C(\theta)\ge0$, choosing a maximizer of $r_{ij}(\theta)$ in every group gives zero total transformed cost, so the disjunct is feasible.  Its LP relaxation value is
\begin{equation}\label{eq:fixed-lp}
L(\theta):=\max\left\{
\sum_{i,j}v_{ij}x_{ij}:
\sum_{i,j}c_{ij}(\theta)x_{ij}\le C(\theta),\quad
\sum_jx_{ij}=1,\ x\ge0
\right\}.
\end{equation}
Define the feasible-threshold index set
$\mathcal F:=\{k\in\{0,\ldots,B-1\}:C(b_k)\ge0\}$ and assume $\mathcal F\ne\varnothing$; if it is empty, every threshold disjunct is infeasible.  The target certificate is
\begin{equation}\label{eq:target-M}
M:=\max_{k\in\mathcal F}L(b_k),
\end{equation}
an upper bound on the integer robust optimum obtained by relaxing every disjunct.  More explicitly, if $Z_k^{\rm IP}$ is the integer optimum of threshold disjunct $k$, then the classical union representation gives
$Z^{\rm robust}=\max_{k\in\mathcal F}Z_k^{\rm IP}\le\max_{k\in\mathcal F}L(b_k)=M$.

Dualize the capacity constraint of \eqref{eq:fixed-lp} with $\lambda\ge0$:
\begin{equation}\label{eq:dense-D}
D(\lambda,\theta):=\lambda C(\theta)
+\sum_i\max_{j\in J_i}\{v_{ij}-\lambda c_{ij}(\theta)\}.
\end{equation}
For a contiguous index interval $I=\{\ell,\ell+1,\ldots,u\}\subseteq\{0,\ldots,B-1\}$, set
\begin{equation}\label{eq:U-I}
U(I):=\inf_{\lambda\ge0}\max_{k\in I\cap\mathcal F}D(\lambda,b_k),
\end{equation}
where $\max\varnothing:=-\infty$.
If interval $I$ has a nonempty finite evaluated multiplier set $\Lambda_I\subseteq[0,\infty)$, define
$\widehat U_{\Lambda_I}(I):=\min_{\lambda\in\Lambda_I}\max_{k\in I\cap\mathcal F}D(\lambda,b_k)$.

\begin{proposition}[Validity, finite-search safety, and singleton exactness]\label{prop:validity}
For every interval $I$,
\[
\max_{k\in I\cap\mathcal F}L(b_k)\le U(I)\le\widehat U_{\Lambda_I}(I).
\]
For a feasible singleton $I=\{k\}$, $U(\{k\})=L(b_k)$.
\end{proposition}
\begin{proof}
If $I\cap\mathcal F=\varnothing$, then every displayed quantity equals $-\infty$ by convention, so the claim is immediate.  Assume henceforth that $I\cap\mathcal F\ne\varnothing$.
For every $k\in\mathcal F$ and $\lambda\ge0$, weak duality gives $L(b_k)\le D(\lambda,b_k)$.  Maximizing over $k\in I\cap\mathcal F$ and then taking the infimum over $\lambda$ gives the first inequality.  Restricting an infimum to $\Lambda_I$ can only increase its value, giving the second.  On a feasible singleton, \eqref{eq:dense-D} is the Lagrangian dual obtained after optimizing over the product of group simplices.  The feasible bounded LP \eqref{eq:fixed-lp} satisfies strong duality, hence $L(b_k)=\inf_{\lambda\ge0}D(\lambda,b_k)$.
\end{proof}

To compare the minimax bound formally with the closest prior-art relaxation, let
$\underline\theta=b_\ell$, $\overline\theta=b_u$, and define $Q(I)$ as the optimum of the bounded-threshold group-clique LP
\begin{align}
\max_{x,p,z}\quad &\sum_{i,j}v_{ij}x_{ij} \label{eq:clique-main}\\
\text{s.t.}\quad
&\sum_{i,j}\left[-a_{ij}+\pos{d_{ij}-\overline\theta}\right]x_{ij}
 +\sum_i p_i+\Gamma z\le-\Gamma\underline\theta,\nonumber\\
&\sum_j\pos{\min\{d_{ij},\overline\theta\}-\underline\theta}x_{ij}
 \le p_i+z &&(i=1,\ldots,n),\nonumber\\
&\sum_jx_{ij}=1 &&(i=1,\ldots,n),\nonumber\\
&x,p\ge0,\qquad 0\le z\le\overline\theta-\underline\theta.\nonumber
\end{align}
This is the exactly-one specialization of the bounded-threshold clique construction of \citet{BusingGersingKoster2023}.  We use the extended-value convention $Q(I)=-\infty$ if this maximization problem is infeasible.

\begin{theorem}[Dominance by the exact minimax envelope]\label{thm:dominance}
For every threshold interval $I$, the exact minimax group-envelope bound is no larger than the bounded-threshold group-clique bound:
\[
\max_{k\in I\cap\mathcal F}L(b_k)\le U(I)\le Q(I).
\]
Thus $U(I)$ is at least as strong as the comparator relaxation for the objective under study, while avoiding an interval LP solve.
\end{theorem}
\begin{proof}
If $I\cap\mathcal F=\varnothing$, then $U(I)=-\infty$ and both inequalities are immediate.  Hence assume $I\cap\mathcal F\ne\varnothing$.
The epigraph formulation of \eqref{eq:U-I} is a linear program.  Its dual has variables $\mu_k\ge0$ and $y_{kij}\ge0$ and can be written
\begin{align*}
\max\quad &\sum_{k\in I\cap\mathcal F}\sum_{i,j}v_{ij}y_{kij}\\
\text{s.t.}\quad
&\sum_{k\in I\cap\mathcal F}\mu_k=1,\qquad
\sum_jy_{kij}=\mu_k &&(k,i),\\
&\sum_{k,i,j}c_{ij}(b_k)y_{kij}
 \le\sum_k\mu_k C(b_k).
\end{align*}
Strong duality applies because the primal epigraph is feasible and has a finite optimum.  The dual can be read as a convex mixture of threshold-indexed group-simplex points coupled by one aggregate capacity inequality.  Take any feasible dual solution and set
\begin{align*}
x_{ij}&:=\sum_k y_{kij},
&z&:=\sum_k\mu_k(b_k-\underline\theta),\\
p_i&:=\sum_{k,j}\pos{\min\{d_{ij},\overline\theta\}-b_k}y_{kij}.
\end{align*}
The simplex equations and bounds on $z$ follow immediately.  For every $b_k\in[\underline\theta,\overline\theta]$,
\[
\pos{\min\{d_{ij},\overline\theta\}-\underline\theta}
\le (b_k-\underline\theta)
+\pos{\min\{d_{ij},\overline\theta\}-b_k},
\]
so the group-clique rows hold after multiplication by $y_{kij}$ and summation.  Expanding the aggregate capacity row of the dual and using $\sum_jy_{kij}=\mu_k$ cancels the group baselines and gives
\[
\sum_{k,i,j}r_{ij}(b_k)y_{kij}\ge
\Gamma\sum_k\mu_k b_k.
\]
The identity
$\pos{d-b_k}=\pos{d-\overline\theta}+\pos{\min\{d,\overline\theta\}-b_k}$
then gives the first constraint of \eqref{eq:clique-main}.  The mapped point has objective $\sum_{k,i,j}v_{ij}y_{kij}$, so every feasible dual value is attained by a feasible clique-LP point.  Maximizing and applying strong duality proves $U(I)\le Q(I)$.
\end{proof}

The finite evaluated bound $\widehat U_{\Lambda_I}(I)$ remains valid by \Cref{prop:validity}, but an arbitrary coarse multiplier set need not inherit \Cref{thm:dominance}.  \Cref{thm:certified-minimax} below closes this gap by quantifying the distance between an explicitly evaluated bound and $U(I)$.

The order of maximum and infimum matters.  Generally
$\max_k\inf_\lambda D(\lambda,b_k)\le\inf_\lambda\max_kD(\lambda,b_k)$, so a nonsingleton interval bound can exceed the complete fixed-threshold scan.  Successive splitting tightens this relaxation, and complete refinement to feasible singletons eliminates it; the evaluation algorithm below removes its computational bottleneck.

\section{Simultaneous group-envelope evaluation}\label{sec:envelope}

\begin{theorem}[Baseline cancellation]\label{thm:cancellation}
For every $\lambda\ge0$ and $\theta\in\B$,
\begin{equation}\label{eq:cancelled-D}
D(\lambda,\theta)
=\sum_i\max_{j\in J_i}\{v_{ij}+\lambda r_{ij}(\theta)\}
-\lambda\Gamma\theta.
\end{equation}
\end{theorem}
\begin{proof}
Substitute $C(\theta)=\sum_i r_i^*(\theta)-\Gamma\theta$ and
$c_{ij}(\theta)=r_i^*(\theta)-r_{ij}(\theta)$ into \eqref{eq:dense-D}.  For each group,
\[
\lambda r_i^*(\theta)+\max_j\{v_{ij}-\lambda[r_i^*(\theta)-r_{ij}(\theta)]\}
=\max_j\{v_{ij}+\lambda r_{ij}(\theta)\}.
\]
Summing and retaining $-\lambda\Gamma\theta$ proves the identity.
\end{proof}

The cancellation is consequential because the remaining threshold dependence has only two forms.  For a group $i$, define
\begin{align}
A_i(\lambda,\theta)&:=\max_{j:d_{ij}\le\theta}\{v_{ij}+\lambda a_{ij}\},\label{eq:A}\\
Q_i(\lambda,\theta)&:=\max_{j:d_{ij}>\theta}\{v_{ij}+\lambda(a_{ij}-d_{ij})\},\label{eq:Q}
\end{align}
where the maximum of an empty set is $-\infty$.

\begin{lemma}[Two-envelope representation]\label{lem:two-envelope}
For $\lambda>0$,
\begin{equation}\label{eq:two-envelope}
\max_j\{v_{ij}+\lambda r_{ij}(\theta)\}
=\max\{A_i(\lambda,\theta),Q_i(\lambda,\theta)+\lambda\theta\}.
\end{equation}
Between consecutive distinct deviations of group $i$, $A_i$ and $Q_i$ are constant.  Hence the group contribution is the maximum of one constant and one line of slope $\lambda$ on that local interval, with at most one crossover.
\end{lemma}
\begin{proof}
If $d_{ij}\le\theta$, then $r_{ij}(\theta)=a_{ij}$ and the option score is constant in $\theta$.  If $d_{ij}>\theta$, then
$r_{ij}(\theta)=a_{ij}-d_{ij}+\theta$, and the score has intercept $v_{ij}+\lambda(a_{ij}-d_{ij})$ and slope $\lambda$.  Maximizing separately over the two sets yields \eqref{eq:two-envelope}.  Those sets are unchanged between consecutive local deviations.
\end{proof}

\subsection{Prefix--suffix range algorithm}

Sort each group by $d_{ij}$.  Before multiplier queries, applying the same unmasked accumulation with $v_{ij}=0$ and $\lambda=1$ computes every $C(b_k)$ and hence $\mathcal F$.  The global threshold indices split into maximal contiguous blocks on which the local saturated set $\{j:d_{ij}\le b_k\}$ is unchanged; all options tied at a block's left endpoint are saturated.  At a fixed $\lambda$, prefix maxima of $v_{ij}+\lambda a_{ij}$ supply the constant values in \eqref{eq:A}, while reverse suffix maxima of $v_{ij}+\lambda(a_{ij}-d_{ij})$ supply the active-line intercepts in \eqref{eq:Q}.  If both envelopes are finite on a block, their crossover is $(A_i-Q_i)/\lambda$; a binary search finds the first global threshold at or above it.  If one envelope is empty, the other applies on the entire block.  Two difference arrays accumulate the resulting constant and linear pieces; cumulative sums recover the group total at all global thresholds.  Finally subtract $\lambda\Gamma b_k$.

\begin{algorithm}[!htbp]
\caption{One simultaneous multiplier evaluation}\label{alg:simultaneous}
\small
\begin{algorithmic}[1]
\Require Global thresholds $b_0<\cdots<b_{B-1}$; deviation-sorted groups; feasibility mask $\mathcal F$; $\lambda\ge0$
\If{$\lambda=0$}
  \State Set every entry to $\sum_i\max_j v_{ij}$
  \State Mask indices outside $\mathcal F$ by $-\infty$; \Return the vector
\EndIf
\State Initialize intercept and slope difference arrays of length $B+1$
\For{each group $i$}
  \State Compute prefix maxima of $v_{ij}+\lambda a_{ij}$
  \State Compute suffix maxima of $v_{ij}+\lambda(a_{ij}-d_{ij})$
  \For{each local deviation interval}
    \State Read its constant $A$ and active-line intercept $Q$
    \If{$A=-\infty$}
      \State Range-add $Q+\lambda b_k$ on the interval
    \ElsIf{$Q=-\infty$}
      \State Range-add $A$ on the interval
    \Else
      \State Locate the clipped first index with $b_k\ge(A-Q)/\lambda$
      \State Range-add $A$ before that index and $Q+\lambda b_k$ from it onward
    \EndIf
  \EndFor
\EndFor
\State Cumulatively sum both arrays and subtract $\lambda\Gamma b_k$
\State Mask indices outside $\mathcal F$; return the $B$ values
\end{algorithmic}
\end{algorithm}

\begin{theorem}[Exact evaluation and complexity]\label{thm:complexity}
Let $K:=\sum_i|J_i|$.  In real arithmetic, after sorting the global thresholds and each group, Algorithm~\ref{alg:simultaneous} returns the values in \eqref{eq:cancelled-D} on $\mathcal F$ and $-\infty$ elsewhere for a fixed multiplier in
\[
O(B+K\log(B+1))\text{ time and }O(B+K)\text{ working storage}.
\]
Constructing the global threshold set and sorting the groups costs $O(K\log K)$.  A direct score evaluation at every option--threshold pair takes $\Theta(BK)$ arithmetic operations; fully materializing those scores also takes $\Theta(BK)$ storage.
\end{theorem}
\begin{proof}
\Cref{lem:two-envelope} proves exactness on every local interval, including repeated deviations and endpoints under the convention $d_{ij}\le\theta$ for the saturated set.  Each group produces at most $|J_i|+1$ local intervals.  Prefix and suffix maxima cost $O(K)$.  Every local interval performs constant work and one binary search in the $B$ global thresholds, totaling $O(K\log(B+1))$.  Moreover, setting all $v_{ij}=0$ and $\lambda=1$ in \eqref{eq:cancelled-D} yields $\sum_i\max_j r_{ij}(b_k)-\Gamma b_k=C(b_k)$, so the same range algorithm constructs the complete feasibility mask within the stated complexity.  The cumulative sums and mask application cost $O(B)$.  The sorted group arrays, local-interval records, and two difference arrays occupy $O(K+B)$ storage.
\end{proof}

\subsection{Certified minimization over the multiplier}

Write $\phi_I(\lambda):=\max_{k\in I\cap\mathcal F}D(\lambda,b_k)$ and define
\begin{equation}\label{eq:lipschitz}
H_I:=\sum_i\max_{j\in J_i}\max\{|a_{ij}|,|a_{ij}-d_{ij}|\}
+\Gamma\max_{k\in I}|b_k|.
\end{equation}
The implementation brackets a minimizer before contracting the bracket.  Starting from zero and a positive scale, it doubles the positive endpoint while the evaluated objective strictly decreases.  Once two consecutive values are nondecreasing, the interval from zero to the latter point contains a minimizer.  Golden-section comparisons then contract this interval.  The algorithm returns the smallest explicitly evaluated value at a multiplier inside the final bracket; it does not use an optimizer-interpolated objective.

\begin{theorem}[Certified deployed minimax bound]\label{thm:certified-minimax}
For every interval with $I\cap\mathcal F\ne\varnothing$, $\phi_I$ is finite, convex, piecewise linear, and $H_I$-Lipschitz on $[0,\infty)$.  The geometric expansion above terminates after finitely many evaluations and produces a bracket $[\ell,u]$ containing a minimizer.  For every evaluated $\bar\lambda\in[\ell,u]$,
\begin{equation}\label{eq:certified-gap}
U(I)\le \bar U(I):=\phi_I(\bar\lambda)
\le U(I)+H_I(u-\ell)\le Q(I)+H_I(u-\ell).
\end{equation}
Consequently, stopping when $H_I(u-\ell)\le\eta_I$ yields a valid evaluated interval bound whose excess over the exact minimax value, and over the clique-dominance benchmark, is at most $\eta_I$.  For a feasible singleton, $L(b_k)\le\bar U(\{k\})\le L(b_k)+\eta_{\{k\}}$.
\end{theorem}
\begin{proof}
For fixed $k$, \eqref{eq:cancelled-D} is a sum of maxima of affine functions of $\lambda$, plus an affine term.  It is therefore finite, convex, and piecewise linear.  A slope of any active affine piece is
$\sum_i r_{ij_i}(b_k)-\Gamma b_k$.  Because $r_{ij}(b_k)$ lies between $a_{ij}-d_{ij}$ and $a_{ij}$, its absolute value is bounded by $H_I$.  A finite pointwise maximum preserves convexity, piecewise linearity, and this common Lipschitz bound, proving the first claim.

In the baseline form \eqref{eq:dense-D}, $c_{ij}(b_k)\ge0$ and every group contains an option with zero transformed cost.  Thus each $D(\cdot,b_k)$ has eventual slope $C(b_k)\ge0$ for $k\in\mathcal F$.  The finite maximum $\phi_I$ is consequently nondecreasing after its last breakpoint.  Geometric expansion therefore reaches consecutive endpoints $a<b$ with $\phi_I(b)\ge\phi_I(a)$.  Convexity then implies that some minimizer lies in $[0,b]$; otherwise all subgradients through $b$ would be negative, contradicting the nonnegative secant slope.  Standard golden-section comparisons preserve at least one minimizer in the retained bracket.

Finally, if $\lambda^*$ is a minimizer in $[\ell,u]$, Lipschitz continuity gives
$0\le\phi_I(\bar\lambda)-\phi_I(\lambda^*)\le H_I|\bar\lambda-\lambda^*|\le H_I(u-\ell)$.
Substitute $\phi_I(\lambda^*)=U(I)$ and apply \Cref{thm:dominance}; singleton exactness follows from \Cref{prop:validity}.
\end{proof}

After bracketing, golden-section contraction needs $O(\log(H_I(u-\ell)/\eta_I))$ simultaneous evaluations.  The released solver targets $\eta_I\le10^{-9}+10^{-8}\max\{1,|\bar U(I)|\}$; at a binary64 multiplier-resolution floor it returns the valid coarser gap rather than claiming the target, while a single feasible threshold is solved by an exact piecewise-linear fallback.  Ragged group arrays realize the stated $O(B+K)$ working storage.  To preserve the real-arithmetic certificate for binary64 inputs, vectorized range accumulation carries a conservative forward-error enclosure, and threshold feasibility is decided by the exact sign of the binary64-represented coefficients whenever that enclosure straddles zero.  Overlapping value enclosures are likewise compared with exact rational arithmetic.  The reported optimality gap includes both the numerical-enclosure width and multiplier-bracket width.  If $A_c$ capacity signs are ambiguous, exact classification adds $O(A_cK)$ work; if $A_v$ multiplier values on interval $I$ are ambiguous, exact comparison adds $O(A_v|I\cap\mathcal F|K)$.  On inputs satisfying the checked aggregate-objective scale condition, the ordinary finite path retains the theorem's arithmetic-operation bound; inputs outside that implementation domain are rejected with rescaling guidance rather than assigned a finite certificate.

\section{Adaptive certification}\label{sec:adaptive}

Let $E$ be the set of thresholds whose fixed LPs have been solved, initialized so that $E\cap\mathcal F\ne\varnothing$; the feasibility mask supplies such an index whenever $\mathcal F\ne\varnothing$.  Let $\mathcal P$ be a collection of active intervals covering every index in $\mathcal F\setminus E$; its intervals may also contain evaluated indices.  Let $U^{\rm disc}$ be the maximum last valid bound of every removed interval not yet dominated by the incumbent, with $U^{\rm disc}=-\infty$ initially.  Maintain
\[
\mathrm{LB}:=\max_{k\in E\cap\mathcal F}L(b_k),
\qquad
\mathrm{UB}:=\max\left\{\mathrm{LB},U^{\rm disc},\max_{I\in\mathcal P}\bar U(I)\right\},
\]
where a maximum over an empty collection is $-\infty$.  The algorithm begins with the root interval, evaluates its endpoints and midpoint, and repeatedly bisects the active interval with largest bound.  The same oracle-independent endpoint--midpoint choices are evaluated in every new nonsingleton interval for both compared methods.  If the three root candidates contain no feasible threshold, the initialization requirement above calls for one additional index from the already constructed feasibility mask before any gap test.  All timed comparison instances had a feasible endpoint or midpoint, so this fallback was not invoked there.  An interval can be removed once its stored bound is within the requested tolerance of $\mathrm{LB}$; its last valid upper bound is retained in $U^{\rm disc}$ until $\mathrm{LB}$ dominates it.  Singleton thresholds are solved directly.

\begin{proposition}[Certificate invariant]\label{prop:certificate}
At every iteration,
$\mathrm{LB}\le M\le\mathrm{UB}$.  If the procedure stops when
\[
\mathrm{UB}-\mathrm{LB}\le\varepsilon\max\{1,|\mathrm{LB}|\},
\]
then it certifies \eqref{eq:target-M} to scaled absolute--relative tolerance $\varepsilon$.  If all feasible thresholds have been evaluated, then $\mathrm{LB}=M$ and all interval records may be removed, setting $\mathrm{UB}:=\mathrm{LB}$.  Alternatively, if every active and retained discarded bound is already at most $\mathrm{LB}$, then the maintained bounds satisfy $\mathrm{LB}=\mathrm{UB}=M$.
\end{proposition}
\begin{proof}
$\mathrm{LB}$ is the maximum over a subset of fixed-threshold LP values.  Every unevaluated feasible threshold belongs to an active interval or to a removed interval whose retained bound is represented by $U^{\rm disc}$; \Cref{prop:validity} therefore implies $M\le\mathrm{UB}$.  Splitting replaces a parent by children whose union is the parent, and direct singleton evaluation transfers its exact LP value into $\mathrm{LB}$.  These operations preserve the inequalities.  The stopping claim follows by division by $\max\{1,|\mathrm{LB}|\}$.  If every feasible threshold has been evaluated, the definition of $M$ gives $M=\mathrm{LB}$ and obsolete interval records can be deleted.  Under the alternative bound-dominance condition, the displayed definition of $\mathrm{UB}$ gives $\mathrm{UB}=\mathrm{LB}$; in either case the stated equality follows.
\end{proof}

The primary certificate concerns $M$, the complete threshold-disjunctive LP upper bound.  The released implementation also provides an exact integer variant: it replaces the LP lower bound by a globally robust integer incumbent, solves a fixed-threshold MCKP branch-and-bound whenever a singleton must be resolved, and prunes a whole interval whenever its retained LP upper bound cannot improve the incumbent.

\begin{corollary}[Exact integer integration]\label{cor:integer-integration}
The integer interval-search variant maintains
$Z^{\rm incumbent}\le Z^{\rm robust}\le\mathrm{UB}$.
If it exhausts the active intervals or terminates with $\mathrm{UB}-Z^{\rm incumbent}$ within the prescribed scaled tolerance, the incumbent is globally optimal to that tolerance.
\end{corollary}
\begin{proof}
Every fixed-threshold integer optimum is bounded by its fixed-threshold LP value, and \Cref{prop:validity} bounds all such LP values in an active interval.  Direct singleton solution transfers an exact integer value into the incumbent; splitting preserves interval coverage; and pruning occurs only when a valid retained upper bound cannot improve the incumbent.  The invariant and stopping claim follow exactly as in \Cref{prop:certificate}.
\end{proof}

\section{Computational study}\label{sec:computations}

\subsection{Design and comparators}

The certified-minimization protocol was serialized before the final rerun; its SHA-256 digest is
\nolinkurl{6db8550e9e9bb47f352bac93423d82067ed09e6ba6cbe4608a7544c2e3a7d261}.
The factorial instance design predates the certified oracle, and the complete final protocol is serialized with the release; it is a reproducible fixed design rather than an external preregistration.  All runs used one thread on an Apple M4 under macOS, Python 3.14.2, NumPy 2.5.1, and SciPy 1.18.0 with HiGHS.  Source data, raw timing repetitions, environments, summaries, and generation scripts accompany the paper.

Four structured families isolate distinct difficulties: \emph{dense frontier} gives many competitive value--resource tradeoffs; \emph{correlated risk} couples value gains and deviations; \emph{near tie} produces unstable rankings; and \emph{many breakpoints} gives essentially unique deviations, so $B$ grows with $K$.  The primary panel uses six options per group, $\Gamma=\lfloor\sqrt n\rfloor$, $n\in\{360,720,1440\}$, five seeds per family--size cell, and five timing repetitions per method.  Method order alternates within an instance; the instance median is the sole statistical observation.  We report geometric-mean paired speedup and a 10,000-draw design-stratified bootstrap interval that resamples seeds within each family--size cell.  These descriptive summaries quantify variation across the fixed generated design; they do not define a superpopulation of robust MCKPs.

The complete-certificate comparator is the bounded-threshold LP of \citet{BusingGersingKoster2023}, specialized with one clique/GUB inequality per exactly-one group as in \eqref{eq:clique-main} and assembled in sparse CSR form.  Both methods use the same best-first splitting policy, fixed-threshold LP routine, endpoint--midpoint candidate rule, scaled tolerance $\varepsilon=10^{-6}$, and time limit including preprocessing.  They differ only in the interval upper bound: the comparator solves a bounded-threshold clique LP, whereas the compressed method evaluates the group-envelope Lagrangian bound.

This is an end-to-end comparison for the stated deliverable---certifying the maximum LP value over the complete threshold family---and a component isolation relative to the broader algorithms of \citet{BusingGersingKoster2023}.  It is not a reconstruction of DnC+: importing that solver would add threshold filtering, estimators, cuts, incumbent operations, early termination, and a different commercial solver stack.  Holding the surrounding search fixed makes the source of any bound, work-count, or timing difference identifiable.

The shared fixed-threshold routine constructs each group's upper value--cost hull and greedily merges its slopes, avoiding branch-and-bound data that neither certificate needs.  Held-out validation compares this LP-only path with an independent reference routine at every threshold; both certificates therefore use the same fixed-MCKP solver.

Protocol-fixed gates required maximum algebraic error at most $2\times10^{-6}$; kernel speedup at least 3$\times$ for $n\ge360$; 100\% primary tolerance completion; geometric-mean primary speedup at least 2$\times$ with interval lower endpoint above 1.5$\times$; at least 80\% timing wins; root dominance in at least 95\% of cases; a median win in all four families; and a median win in every robustness configuration.

\subsection{Correctness validation and oracle ablations}

Forty held-out irregular instances use unequal menu sizes, repeated and near-repeated deviations, signed objectives, and $\Gamma$ values from 0 to $n$.  We compare the compressed trace with (i) the dense baseline-cost trace, (ii) direct evaluation of \eqref{eq:cancelled-D}, (iii) independently optimized interval bounds, (iv) independently solved epigraph LPs for randomly selected intervals, (v) an independent fixed-MCKP LP routine at every threshold, and (vi) a complete threshold scan.  The maximum algebraic discrepancy is \PubValidationMaxError; the maximum violation of the certified upper/lower enclosure against the independent epigraph LP is \PubValidationCertViolation; and the largest scaled minimization certificate is \PubValidationCertGap.  The minimum root-bound slack over the complete scan is \PubValidationMinSlack.  All serialized validation gates pass.

The kernel ablation separates algebraic compression from adaptive search.  It contains \PubKernelCases{} instances, nine normalized multiplier queries, five repetitions, and balanced order.  \Cref{tab:kernel} reports geometric-mean query and total construction-plus-query speedups.  The maximum identity error is \PubKernelMaxError.  For $n\ge360$, total speedup is \PubKernelLargeSpeedup$\times$ in geometric mean.  In the many-breakpoints family, where $B$ grows linearly with $K$, the descriptive log--log slopes are \PubKernelCompressedSlope{} for the compressed evaluator and \PubKernelDenseSlope{} for dense materialization; the left panel of \Cref{fig:scaling} shows the separation.  Storage denotes persistent numerical-array bytes rather than peak process memory and grows sharply only in that family.

To isolate the interval oracles from adaptive-tree effects, a common-trace ablation evaluates both bounds on the same \PubTraceIntervals{} dyadic intervals from \PubTraceCases{} independently generated instances.  The envelope oracle wins \PubTraceWins{} instance timings with a \PubTraceGeoSpeedup$\times$ geometric-mean speedup; every evaluated bound is certified, with maximum scaled minimization gap \PubTraceCertGap, and its bound is no larger than the clique bound on \PubTraceDominancePct{} of matched intervals.  This comparison includes preprocessing and sparse LP assembly.  Because the intervals are identical, it is the direct isolation test of the interval-bound component; \Cref{fig:common-trace} shows that the advantage is present in every family and instance.  The adaptive experiment below then measures the additional effect of tighter bounds on the search trajectory.

\begin{figure}[t]
\centering
\includegraphics[width=0.82\textwidth]{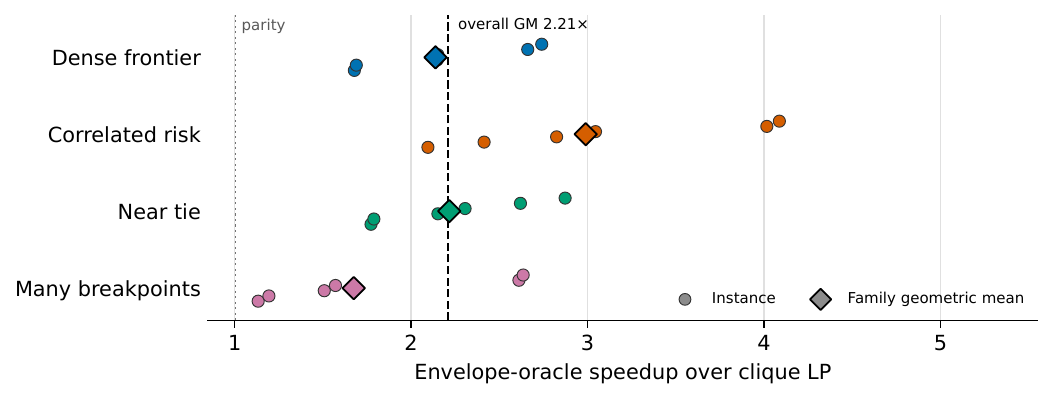}
\caption{Matched-trace comparator isolation.  Each circle is one independently generated instance; vertical offsets only separate observations within a family, and diamonds mark family geometric means.  The dotted line marks parity and the dashed line the overall \PubTraceGeoSpeedup$\times$ geometric mean.  All \PubTraceCases{} observations favor the envelope oracle, and its certified bound is no larger than the clique bound on all \PubTraceIntervals{} matched intervals.}
\label{fig:common-trace}
\end{figure}

\begin{table}[t]
\centering
\caption{Kernel ablation against dense threshold materialization.  Speedups are geometric means over four families and three seeds.  Storage is the median (maximum) ratio of persistent numerical-array bytes and excludes interpreter overhead.}
\label{tab:kernel}
\small
\begin{tabular}{rrrrrr}
\toprule
Groups & Cases & Query speedup & Total speedup & Storage ratio & Maximum error \\
\midrule
90 & 12 & 3.48 & 33.63 & 2.62 (35.7) & 6.7e-09 \\
180 & 12 & 5.82 & 44.68 & 2.64 (71.2) & 2.7e-08 \\
360 & 12 & 8.53 & 56.19 & 2.65 (142.2) & 1.1e-07 \\
720 & 12 & 11.49 & 64.37 & 2.65 (284.3) & 4.3e-07 \\
\bottomrule
\end{tabular}

\end{table}

\subsection{Primary result}

Both methods reach tolerance on all \PubPrimaryCases{} primary instances.  Every envelope bound is certified, with maximum scaled minimization gap \PubPrimaryCertGap.  The envelope method wins all \PubPrimaryWins{} paired timings, with geometric-mean speedup \PubPrimaryGeoSpeedup$\times$ and design-stratified 95\% interval [\PubPrimaryCILow, \PubPrimaryCIHigh].  It evaluates a median \PubPrimaryMedianThetaPct{} of thresholds as fixed LPs.  Final lower bounds are identical at recorded precision, and its root bound never exceeds the clique root bound within numerical tolerance.  At the raw repeat-block level, \PubRepeatWins{}/\PubRepeatBlocks{} are wins (minimum speedup \PubRepeatMinSpeedup$\times$); median within-instance coefficients of variation are \PubCompressedMedianCV{} and \PubCliqueMedianCV{}.  Across an instance's interval solves, the clique comparator uses a median \PubCliqueMedianNnz{} sparse-matrix nonzeros and \PubCliqueMedianCSRMiB{} MiB of cumulative CSR arrays.  Together with the matched-interval result, these observations show that the gain survives complete certification and is not produced by different fixed-threshold solvers or search policies.  The right panel of \Cref{fig:scaling} displays instance- and size-level scaling.

\begin{table}[t]
\centering
\caption{Primary complete-certificate comparison.  Each reported time and evaluation count is the median across instances of that size after taking the within-instance median over five balanced timing repetitions.  Aggregate speedup treats each instance once.  In the fixed-threshold evaluation column, C denotes the compressed certificate and Q the clique certificate.}
\label{tab:primary}
\small
\resizebox{\textwidth}{!}{\begin{tabular}{rrrrrrrr}
\toprule
Groups & Cases & Compressed (s) & Clique (s) & Geometric speedup & Wins & \shortstack{Fixed-threshold LP\\evaluations (C / Q)} & \shortstack{Clique interval LP\\evaluations} \\
\midrule
360 & 20 & 0.083 & 0.140 & 1.69 & 20/20 & 12/12 & 8 \\
720 & 20 & 0.140 & 0.323 & 2.42 & 20/20 & 12/12 & 8 \\
1440 & 20 & 0.255 & 0.762 & 3.24 & 20/20 & 12/12 & 8 \\
\bottomrule
\end{tabular}
}
\end{table}

\begin{figure}[t]
\centering
\includegraphics[width=0.92\textwidth]{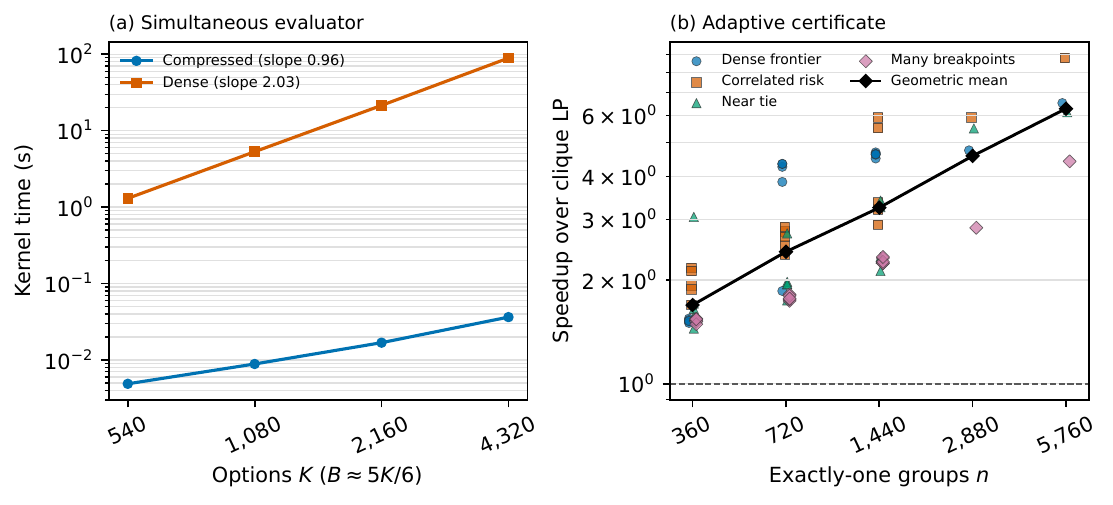}
\caption{Mechanism and end-to-end scaling.  Left: construction plus nine multiplier queries in the many-breakpoints kernel panel, where $B$ grows linearly with $K$; lines show geometric means over three seeds, with descriptive log--log slopes.  Right: instance-level speedup of the compressed adaptive certificate over the clique-LP certificate; black diamonds are geometric means by size.  Primary observations use five seeds at $n\le1440$ and stress observations use one separate seed at $n\in\{2880,5760\}$.}
\label{fig:scaling}
\end{figure}

Every family median exceeds one, from \PubFamilyBreakpoints$\times$ for many breakpoints to \PubFamilyDense$\times$ for dense frontier, so the protocol-fixed breadth gate passes.

\subsection{Robustness, stress, and published-coefficient evidence}

The 36-instance robustness panel changes one dimension at $n=720$: $\Gamma=1$, twelve-option menus, or $\Gamma=\lfloor0.2n\rfloor$.  Every instance reaches tolerance, and every configuration-level bootstrap interval remains above one.  The wide-menu and large-budget panels each record 12/12 timing wins; the sparse-budget panel records 6/12, identifying the regime in which the clique search already eliminates most thresholds and leaves less interval-bound work to compress.

\begin{table}[t]
\centering
\caption{Robustness configurations at 720 groups, four families, and three new seeds.  Within each configuration, bootstrap intervals resample the three seeds separately by family.}
\label{tab:robustness}
\small
\resizebox{\textwidth}{!}{\begin{tabular}{lrrr}
\toprule
Configuration & Geometric speedup [95\% interval] & Wins & Thresholds evaluated \\
\midrule
Sparse budget ($m=6,\Gamma=1$) & 1.29 [1.20, 1.46] & 6/12 & 11.5\% \\
Wide menu ($m=12,\Gamma=\lfloor\sqrt n\rfloor$) & 2.59 [2.28, 2.95] & 12/12 & 28.8\% \\
Large budget ($m=6,\Gamma=0.2n$) & 2.83 [2.66, 3.02] & 12/12 & 34.6\% \\
\bottomrule
\end{tabular}
}
\end{table}

The descriptive stress panel uses four families, one new seed, and two repetitions at each of 2,880 and 5,760 groups.  All eight instances reach tolerance with identical final lower bounds.  The geometric-mean speedup is \PubStressGeoSpeedup$\times$ overall and \PubStressNHighSpeedup$\times$ at 5,760 groups.  The many-breakpoints instances contain 14,401 and 28,801 thresholds, respectively; detailed rows are in the companion.

For an out-of-generator check, we use the published robust-knapsack archive of \citet{GersingData2022}.  Each binary item becomes the two-option group given after \eqref{eq:robust-mckp}.  We retain the archive's nominal profits, weights, budgets, and deviation vectors, but transfer its objective deviations to resource deviations; this is a transparent model-compatible coefficient stress test, not a reproduction of the source uncertain-objective model.  The panel contains three published seeds at each of 1,000, 5,000, and 10,000 items.  Both methods reach tolerance, every envelope bound is certified with scaled gap at most \PubExternalCertGap, and the method wins \PubExternalWins{}/\PubExternalCases{} timings with geometric-mean speedup \PubExternalGeoSpeedup$\times$, rising to \PubExternalNHighSpeedup$\times$ at 10,000 items.

\FloatBarrier

\subsection{Exact integer integration audit}

The integer integration in \Cref{cor:integer-integration} was audited separately from the LP-certificate experiment.  Twelve instances use 30 groups, all four families, three new seeds, two balanced timing repetitions, and a five-second limit.  The envelope search, clique search, and complete threshold enumeration each certify 7 of 12 instances; compact SCIP certifies all 12.  Every jointly certified objective agrees exactly at recorded precision.  On the seven jointly certified envelope--clique cases, the geometric-mean clique/envelope time ratio is 0.61, showing that fixed-threshold integer search, rather than interval-bound computation, dominates at this scale.  Both interval methods and enumeration certify all dense-frontier and near-tie cases and one of three correlated-risk cases; none certifies a many-breakpoints case within five seconds, while compact SCIP certifies all three.

This result is a bottleneck decomposition rather than counterevidence to the LP-certificate result.  The envelope oracle reduces the interval-bound component; it cannot accelerate fixed-threshold branch-and-bound once that component controls total time.  The audit therefore validates the exact integration and locates its operating boundary without supporting universal integer-solver superiority.  The demonstrated computational claim remains acceleration of the complete LP-bound certificate, particularly in the root and outer-search regime where bounded-threshold relaxations are material.

\subsection{Application-derived semi-synthetic panel}

The pricing model in \Cref{sec:pricing-specialization} also supplies an application-derived realism check.  We use the UCI Online Retail data set, which contains 541,909 transaction records \citep{Chen2015UCI}.  The released calibration reads the first 200,000 records; after positive-price/quantity filtering and the eight-observation SKU rule, 192,451 observations contribute to 2,549 SKU aggregates.  These aggregates calibrate the segment mix, price levels, volumes, and uncertainty scales of nine semi-synthetic pricing portfolios; documented elasticity and menu-generation assumptions remain modeled rather than observed.  For each portfolio, the reported certificate is an upper bound on the best robust-feasible nominal revenue across the complete price-risk threshold family.  The panel uses 360, 720, and 1,440 groups, twelve price choices, three new seeds per size, and three balanced repetitions.  Both methods reach tolerance, but the compressed method wins \PubApplicationWins{}/\PubApplicationCases{} instances and the clique/compressed time ratio is \PubApplicationGeoSpeedup$\times$ (95\% cell-stratified bootstrap interval [\PubApplicationCILow, \PubApplicationCIHigh]); even at 1,440 groups it is \PubApplicationNHighSpeedup$\times$.  The search evaluates a median 0.5\% of thresholds as fixed LPs, so little bounded-threshold work remains to amortize certified minimization.  This panel is therefore a useful negative boundary as well as a coefficient-scale check.  Demand estimation is outside the algorithmic scope, and the nonrandom row cap and modeled demand curves preclude causal or transaction-level validation claims.

\subsection{Scope and limitations}

The experiments validate the envelope identity, certified minimax approximation, retained LP certificate, and speed advantage over the tested sparse clique-LP implementation in the primary, stress, and published-coefficient designs, not universal dominance.  The sparse-budget results are mixed, and the UCI-calibrated panel reverses the timing advantage when the search evaluates only a very small threshold fraction.  The published-data panel transfers coefficients rather than reproducing the source uncertain-objective problem, and the pricing panel remains semi-synthetic.  The implementation uses Python/SciPy on one architecture, while compiled code could shift constants; streaming could reduce dense-oracle storage without changing its $\Theta(BK)$ arithmetic.  As established in the introduction and comparator design, the official implementation of \citet{BusingGersingKoster2023} targets objective uncertainty and combines several algorithmic components with Java/Gurobi.  The controlled comparison therefore isolates its closest bounded-threshold clique relaxation rather than claiming a direct reproduction of DnC+.

The exact audit shows that fixed-threshold integer search can erase the interval-oracle advantage and that compact SCIP is stronger on the tested many-breakpoints instances.  The UCI panel similarly shows that aggressive threshold elimination can leave too little LP-family work to amortize the certified oracle.  The study directly supports complete LP-family certification and interval bounding when that work is material in the tested adaptive search.  Repeated node-level bounds, pricing-loop re-solves, and sensitivity sweeps are plausible integration settings but are not evaluated here.  DnC+, specialized fixed-MCKP solvers, and the oracle are therefore complementary.

\section{Conclusion}\label{sec:conclusion}

For exactly-one robust models, the proposed oracle makes the complete threshold-disjunctive LP certificate a more practical root or outer-search bound.  After sorting, group envelopes evaluate one multiplier over all thresholds in $O(B+K\log(B+1))$ time and $O(B+K)$ working storage.  Convex bracketing converts those evaluations into an explicit minimization certificate: the deployed bound is valid and reports its distance certificate to the exact minimax value, reaches the prescribed tolerance unless multiplier resolution intervenes, is singleton-exact through an exact fallback, and is no larger than the clique bound plus its reported gap.

The evidence is both component-matched and end-to-end for this deliverable.  On identical interval traces the proposed oracle is no weaker throughout and faster on every instance; under the shared adaptive search, the study records zero certificate violations, \PubPrimaryWins{} of \PubPrimaryCases{} primary wins, a \PubPrimaryGeoSpeedup$\times$ geometric-mean speedup, and \PubExternalWins{} of \PubExternalCases{} wins on published knapsack coefficients.  The exact integration locates the boundary: integer subproblem work can dominate, and compact SCIP remains stronger on many-breakpoints instances.  Integration into full robust-solver stacks is therefore the natural computational extension, while multiple robust constraints remain the clearest theoretical extension.

\section*{Data and code availability}
\ifblind
The anonymous supplement contains source code, tests, the protocol and digest, instance generators, raw timing repetitions, result tables, environment records, and the script that regenerates every numerical artifact.
\else
Code and reproducibility artifacts are available at
\url{https://github.com/eric939/simultaneous-group-envelope-mckp}.  The evidence directory contains the serialized protocol, per-phase environments, raw CSV files, JSON summaries, and the SHA-256 artifact manifest.
\fi

\section*{Acknowledgments}
OpenAI ChatGPT and Codex assisted with drafting, editing, code generation, and computational and mathematical auditing.  The author reviewed and verified the work and assumes full responsibility.

\fi

\ifincludeappendix
\FloatBarrier
\appendix
\numberwithin{table}{section}
\numberwithin{figure}{section}

\section{Comparator and exact-integration implementation}\label{app:implementation}

The main manuscript states the bounded-threshold group-clique comparator and proves its formal relationship to the minimax bound.  Both of its constraint matrices are passed to HiGHS in sparse CSR form.  An ambiguous generic HiGHS status is never accepted as an upper bound: the code retries dual simplex and interior point and accepts only optimizer-certified optimality or infeasibility.

The exact integration uses the same best-first interval queue for the envelope and clique variants.  A globally robust HullRound solution initializes the incumbent.  The root endpoints and midpoint are solved by the fixed-threshold branch-and-bound; afterward the algorithm bisects the active interval with largest retained bound.  A singleton is solved exactly only when its interval remains capable of improving the incumbent.  The complete-enumeration comparator uses the same fixed-threshold branch-and-bound ordered by fixed-threshold LP value, while compact SCIP uses the standard budgeted-uncertainty reformulation.  The audit contains four families, 30 groups, three seeds, two balanced repetitions, and a five-second limit.  Raw results, solver statuses, node counts, interval counts, objective agreement, and environment metadata are included in the evidence directory.

\section{Detailed protocol and statistical estimands}\label{app:protocol}

The statistical unit is an instance.  Timing repetitions estimate the runtime of that instance and are not treated as independent observations.  For primary instance $r\in\{1,\ldots,R\}$, let $t_r^C$ and $t_r^Q$ be the medians of five wall-clock repetitions for the compressed and clique methods.  The paired speedup is $s_r=t_r^Q/t_r^C$, and the reported aggregate is
\[
\exp\left(\frac1R\sum_{r=1}^R\log s_r\right).
\]
The percentile interval resamples seeds with replacement separately inside each family--size cell for 10,000 draws, preserving the factorial design.  It is a design-stratified descriptive measure of variation across generated seeds, not a population-sampling interval for all robust MCKPs.  No timing repetition enters this calculation.

The primary design has 60 instances: three sizes, four families, and five seeds.  Robustness has 36 new instances: three configurations, four families, and three seeds at $n=720$.  Stress has eight new instances: two sizes, four families, and one seed.  The kernel has 48 instances: four sizes, four families, and three seeds.  Validation has 40 independently generated irregular cases.  Two separately scoped nine-instance panels use UCI-calibrated semi-synthetic portfolios and published robust-knapsack coefficients, respectively.  Thread environment variables were fixed to one for every phase.

\section{Exact instance generators}\label{app:generators}

For auditability, we state the structured generator.  Each group has a base option and $m-1$ upgrades.  Let $q_k$, $k=1,\ldots,m-1$, be equally spaced from 0.8 to 14, let $w_i\sim U[8,8.4]$, and let
$G$ be the 25-point grid from 0.2 to 3.2.  The base option has
$(v_{i0},a_{i0},d_{i0})=(w_i,b_f,0)$, where
$b_f$ is 4.6, 4.1, 3.9, and 4.4 for dense frontier, correlated risk, near tie, and many breakpoints, respectively.  Upgrade margins are
\[
a_{ik}=b_f-q_k+\eta_{ik},\qquad \eta_{ik}\sim N(0,0.08^2).
\]
Independent value noises give
\[
v_{ik}=\max\{0,w_i+h_f(q_k)+\epsilon_{ik}\},
\]
with $(h_f(q),\operatorname{sd}\epsilon)$ equal to
$(8\sqrt q,0.35)$, $(3.15q,0.75)$, $(2.25q,1.65)$, and
$(7.5\sqrt q,0.45)$ in the same family order.  With zero-based group index $i$, deviations are
\begin{align*}
d_{ik}^{\rm dense}&=G[(7i+3k+s)\bmod25],\\
d_{ik}^{\rm corr}&=\operatorname*{nearest}_{g\in G}
\{0.30+0.19q_k+\zeta_{ik}\},\quad \zeta_{ik}\sim N(0,0.08^2),\\
d_{ik}^{\rm tie}&=G[(5i+k+s)\bmod25],\\
d_{ik}^{\rm break}&=0.15+0.003[i(m-1)+k]+10^{-6}s,
\end{align*}
where $s$ is the reported seed.  To reproduce the release exactly, the NumPy seed is the first 16 hexadecimal digits of
\texttt{SHA256(`v4|family|n|m|Gamma|seed')} reduced modulo $2^{32}$.  This seed derivation is platform-independent; the released environment record and result files remain authoritative because numerical distribution routines can vary across library versions.

The irregular validation generator independently draws 8--20 groups and 2--10 options per group, signed Gaussian values, Gaussian margins, repeated deviations from $\{0,0.25,0.5,1,2,4,8\}$ with occasional perturbations below 0.05, and $\Gamma$ uniformly from $\{0,\ldots,n\}$.  These cases target algebraic edge conditions rather than timing difficulty.

For the application-derived panel, the calibration reads the first 200,000 UCI records and aggregates positive-quantity, positive-price observations by stock code.  Products with fewer than eight retained observations are removed, leaving 192,451 contributing observations and 2,549 SKU aggregates.  Price--volume quantiles determine five segment shares and reference medians.  Within each generated portfolio, segment-calibrated lognormal prices and volumes combine with documented elasticity ranges; twelve psychologically rounded price choices induce objective, nominal-resource, and uncertainty coefficients.  The released segment aggregate has SHA-256 digest \nolinkurl{1ba04343eb853439f184a55dbd6ac3dbeac6f6438ef1b9dcc8b0173e3a40b7f6}; the raw data are retrieved from the cited UCI record.

The published-coefficient panel reads the CC-BY robust-knapsack archive of \citet{GersingData2022}, verified by SHA-256 digest \nolinkurl{8571b3e545607415a38a39dc506b21bd891b6a22ce252e42a1622a5a5f451818}.  For each binary item, the parser creates an unselected and selected option using the exact two-option reduction stated in the main paper.  The source nominal profit, weight, $\Gamma$, and deviation coefficient are otherwise unchanged.  Because the source deviations perturb objective coefficients whereas the model studied here perturbs a resource row, this panel is labeled coefficient transfer throughout; its purpose is generator independence, not source-model replication.

\begin{center}
\centering
\captionof{table}{Published robust-knapsack coefficient-transfer panel.  Times are instance medians over two balanced repetitions.  Source coefficients and budgets are retained, while deviations are transferred from the source uncertain objective to the resource constraint.}
\label{tab:external}
\small
\begin{tabular}{rrrrrr}
\toprule
Binary items & Cases & Envelope (s) & Clique (s) & Geometric speedup & Wins \\
\midrule
1,000 & 3 & 0.319 & 0.529 & 1.69 & 3/3 \\
5,000 & 3 & 1.292 & 4.165 & 3.07 & 3/3 \\
10,000 & 3 & 2.636 & 10.986 & 4.00 & 3/3 \\
\bottomrule
\end{tabular}

\end{center}

\section{Reproducibility checklist}\label{app:reproducibility}

The supplement contains the protocol and digest, phase-specific environments, instance-level and raw timing CSVs, summaries with gates, exact-epigraph certificate checks, unit tests for every algebraic and certificate invariant, and a generator for all numerical manuscript artifacts and their hash manifest.  Kernel storage means persistent numerical-array bytes rather than process RSS; comparator matrix storage is measured from sparse CSR arrays; the two-repetition stress and published-coefficient panels are descriptive.

\fi


\begin{thebibliography}{99}
\footnotesize
\setlength{\bibsep}{0pt}

\bibitem[\'{A}lvarez-Miranda et~al.(2013)]{AlvarezMiranda2013}
\'{A}lvarez-Miranda, E., Ljubi\'{c}, I., and Toth, P. (2013).
\newblock A note on the Bertsimas \& Sim algorithm for robust combinatorial optimization problems.
\newblock \emph{4OR}, 11, 349--360.
\newblock \href{https://doi.org/10.1007/s10288-013-0231-6}{doi:10.1007/s10288-013-0231-6}.

\bibitem[Atamt\"urk(2006)]{Atamturk2006}
Atamt\"urk, A. (2006).
\newblock Strong formulations of robust mixed 0--1 programming.
\newblock \emph{Mathematical Programming}, 108, 235--250.
\newblock \href{https://doi.org/10.1007/s10107-006-0709-5}{doi:10.1007/s10107-006-0709-5}.

\bibitem[Bertsimas and Sim(2003)]{BertsimasSim2003}
Bertsimas, D. and Sim, M. (2003).
\newblock Robust discrete optimization and network flows.
\newblock \emph{Mathematical Programming}, 98, 49--71.
\newblock \href{https://doi.org/10.1007/s10107-003-0396-4}{doi:10.1007/s10107-003-0396-4}.

\bibitem[Bertsimas and Sim(2004)]{BertsimasSim2004}
Bertsimas, D. and Sim, M. (2004).
\newblock The price of robustness.
\newblock \emph{Operations Research}, 52, 35--53.
\newblock \href{https://doi.org/10.1287/opre.1030.0065}{doi:10.1287/opre.1030.0065}.

\bibitem[B\"using et~al.(2023)]{BusingGersingKoster2023}
B\"using, C., Gersing, T., and Koster, A.~M.~C.~A. (2023).
\newblock A branch and bound algorithm for robust binary optimization with budget uncertainty.
\newblock \emph{Mathematical Programming Computation}, 15, 269--326.
\newblock \href{https://doi.org/10.1007/s12532-022-00232-2}{doi:10.1007/s12532-022-00232-2}.

\bibitem[Caserta and Vo{\ss}(2019)]{Caserta2019}
Caserta, M. and Vo{\ss}, S. (2019).
\newblock The robust multiple-choice multidimensional knapsack problem.
\newblock \emph{Omega}, 86, 16--27.
\newblock \href{https://doi.org/10.1016/j.omega.2018.06.014}{doi:10.1016/j.omega.2018.06.014}.

\bibitem[Chen(2015)]{Chen2015UCI}
Chen, D. (2015).
\newblock Online Retail [Dataset].
\newblock \emph{UCI Machine Learning Repository}.
\newblock \href{https://doi.org/10.24432/C5BW33}{doi:10.24432/C5BW33}.

\bibitem[Dyer(1984)]{Dyer1984}
Dyer, M.~E. (1984).
\newblock An $O(n)$ algorithm for the multiple-choice knapsack linear program.
\newblock \emph{Mathematical Programming}, 29, 57--63.
\newblock \href{https://doi.org/10.1007/BF02591729}{doi:10.1007/BF02591729}.

\bibitem[Fischetti and Monaci(2012)]{FischettiMonaci2012}
Fischetti, M. and Monaci, M. (2012).
\newblock Cutting plane versus compact formulations for uncertain (integer) linear programs.
\newblock \emph{Mathematical Programming Computation}, 4, 239--273.
\newblock \href{https://doi.org/10.1007/s12532-012-0039-y}{doi:10.1007/s12532-012-0039-y}.

\bibitem[Gersing et~al.(2022)]{GersingData2022}
Gersing, T., B\"using, C., and Koster, A.~M.~C.~A. (2022).
\newblock Benchmark instances for robust combinatorial optimization with budgeted uncertainty.
\newblock \emph{Zenodo}.
\newblock \href{https://doi.org/10.5281/zenodo.7419028}{doi:10.5281/zenodo.7419028}.

\bibitem[Hansknecht et~al.(2018)]{Hansknecht2018}
Hansknecht, C., Richter, A., and Stiller, S. (2018).
\newblock Fast robust shortest path computations.
\newblock \emph{OASIcs ATMOS}, 65, 5:1--5:21.
\newblock \href{https://doi.org/10.4230/OASIcs.ATMOS.2018.5}{doi:10.4230/OASIcs.ATMOS.2018.5}.

\bibitem[Joung and Park(2021)]{JoungPark2021}
Joung, S. and Park, K. (2021).
\newblock Robust mixed 0--1 programming and submodularity.
\newblock \emph{INFORMS Journal on Optimization}, 3, 183--199.
\newblock \href{https://doi.org/10.1287/ijoo.2019.0042}{doi:10.1287/ijoo.2019.0042}.

\bibitem[Joung et~al.(2023)]{JoungOhLee2023}
Joung, S., Oh, S., and Lee, K. (2023).
\newblock Comparative analysis of linear programming relaxations for the robust knapsack problem.
\newblock \emph{Annals of Operations Research}, 323, 65--78.
\newblock \href{https://doi.org/10.1007/s10479-022-05161-w}{doi:10.1007/s10479-022-05161-w}.

\bibitem[Lee and Kwon(2014)]{LeeKwon2014}
Lee, T. and Kwon, C. (2014).
\newblock A short note on robust combinatorial optimization problems with cardinality constrained uncertainty.
\newblock \emph{4OR}, 12, 373--378.
\newblock \href{https://doi.org/10.1007/s10288-014-0270-7}{doi:10.1007/s10288-014-0270-7}.

\bibitem[Monaci et~al.(2013)]{Monaci2013}
Monaci, M., Pferschy, U., and Serafini, P. (2013).
\newblock Exact solution of the robust knapsack problem.
\newblock \emph{Computers \& Operations Research}, 40, 2625--2631.
\newblock \href{https://doi.org/10.1016/j.cor.2013.05.005}{doi:10.1016/j.cor.2013.05.005}.

\ifblind
\bibitem[Anonymous(2026)]{Shao2026PaperA}
Anonymous (2026).
\newblock Companion paper on certifying finite-menu robust pricing; identifying details withheld for double-blind review.
\else
\bibitem[Shao(2026)]{Shao2026PaperA}
Shao, Z. Y. E. (2026).
\newblock A certifying MCKP framework for $\Gamma$-robust discrete pricing.
\newblock \emph{arXiv preprint arXiv:2603.18653}, version 2.
\newblock \href{https://arxiv.org/abs/2603.18653}{arXiv:2603.18653}.
\fi

\bibitem[Sinha and Zoltners(1979)]{SinhaZoltners1979}
Sinha, P. and Zoltners, A.~A. (1979).
\newblock The multiple-choice knapsack problem.
\newblock \emph{Operations Research}, 27, 503--515.
\newblock \href{https://doi.org/10.1287/opre.27.3.503}{doi:10.1287/opre.27.3.503}.

\bibitem[Szkaliczki(2025)]{Szkaliczki2025}
Szkaliczki, T. (2025).
\newblock Solution methods for the multiple-choice knapsack problem and their applications.
\newblock \emph{Mathematics}, 13, 1097.
\newblock \href{https://doi.org/10.3390/math13071097}{doi:10.3390/math13071097}.

\bibitem[Zemel(1980)]{Zemel1980}
Zemel, E. (1980).
\newblock The linear multiple choice knapsack problem.
\newblock \emph{Operations Research}, 28, 1412--1423.
\newblock \href{https://doi.org/10.1287/opre.28.6.1412}{doi:10.1287/opre.28.6.1412}.

\end{thebibliography}
\end{document}